\documentclass[12pt]{amsart}
\usepackage{amsmath,enumerate}
\usepackage{amssymb}
\newtheorem{theorem}{Theorem}
\newtheorem{prop}[theorem]{Proposition}
\newtheorem{defn}[theorem]{Definition}
\newtheorem{lemma}[theorem]{Lemma}
\newtheorem{prob}[theorem]{Problem}

\newtheorem{remark}[theorem]{Remark}

\newcommand{\C}{\Bbb C}
\newcommand{\D}{\Bbb D}
\newcommand{\G}{\Bbb G}
\newcommand{\N}{\Bbb N}

\newcommand{\Om}{\Omega}

\newcommand{\hol}{{\mathrm {Hol}}}
\newenvironment{proof*}{\vskip 2mm\noindent {}}{\hfill $\Box$ \vskip 2mm}
\def\O{\mathcal O}
\def\M{\mathcal M}

\def\ord{\operatorname{ord}}
\def\sp{\operatorname{Sp}}

\title[Lifting maps]{Lifting maps from the symmetrized polydisc in small dimensions}

\begin{document}

\author{Nikolai Nikolov}
\address{Institute of Mathematics and Informatics\\ Bulgarian Academy of
Sciences\\1113 Sofia, Bulgaria} \email{nik@math.bas.bg }

\author{Pascal J. Thomas}
\address{Universit\'e de Toulouse\\ UPS, INSA, UT1, UTM \\
Institut de Math\'e\-ma\-tiques de Toulouse\\
F-31062 Toulouse, France} \email{pascal.thomas@math.univ-toulouse.fr}

\author{Tran Duc-Anh}
\address{Department of Mathematics \\
Hanoi National University of Education\\
136 Xuan Thuy St., Hanoi, Vietnam} \email{ducanh@hnue.edu.vn}



\begin{abstract}
The spectral unit ball $\Omega_n$ is the set of all $n\times n$ matrices $M$ with
spectral radius less than $1$. Let  $\pi(M) \in \C^n$
stand for the coefficients of the characteristic polynomial of
a matrix $M$ (up to signs),
i.e. the elementary symmetric functions of its eigenvalues.  The
symmetrized polydisc is $\mathbb G_n:=\pi(\Omega_n)$.

When investigating Nevanlinna-Pick problems for maps from the disk to
the spectral ball, it is often useful to project the map to the symmetrized polydisc
(for instance to obtain continuity results for the Lempert function): if $\Phi \in \hol (\D, \Omega_n)$, then $\pi \circ \Phi \in \hol (\D, \mathbb G_n)$.  Given a map $\varphi \in \hol(\D, \mathbb G_n)$,
we are looking for necessary and sufficient conditions for this map to ``lift
through given matrices", i.e. find $\Phi$ as above so that $\pi \circ \Phi = \varphi$
and $\Phi (\alpha_j) = A_j$, $1\le j \le N$.
A natural necessary condition is $\varphi(\alpha_j)=\pi(A_j)$, $1\le j \le N$.
When the matrices $A_j$ are derogatory (i.e. do not admit a cyclic vector)
new necessary conditions appear, involving derivatives of $\varphi$ at the points $\alpha_j$.
We prove that those conditions are necessary and sufficient for a local lifting.
We give a formula which performs the global lifting in small dimensions ($n \le 5$),
and a counter-example to show\textsl{•}
that the formula fails in dimensions $6$ and above.
\end{abstract}

\maketitle

\section{Motivation and statements}
\subsection{Definitions}
Some problems in Robust Control Theory lead to the study of
structured singular values of a matrix (denoted by $\mu$).  A special case of
this is simply the spectral radius. A very special instance of the ``$\mu$-synthesis"
problem reduces to a
Nevanlinna-Pick problem, i.e. given points $\alpha_j \in \D
:=\{z\in \C: |z|<1\}$, $A_j \in \Omega \subset \C^m$,
$1\le j \le N$, determine whether there exists $\Phi$ holomorphic from $\D$ to  $\Omega $
such that $\Phi(\alpha_j) = A_j$, $1\le j \le N$. We refer the interested
reader to Nicholas Young's stimulating survey \cite{Yo}.

We study this special case.  Let us
set some notation.

Let $\M_n$ be the set of all $n\times n$ complex matrices. For
$A\in\M_n$ denote by $\sp(A)$ and $r(A)=\max_{\lambda\in
\sp(A)}|\lambda|$ the spectrum and the spectral radius of $A$,
respectively.

\begin{defn}
\label{balldisk}
The \emph{spectral ball} $\Om_n$ is given as
$$
\Om_n:=\{A\in\M_n:r(A)<1\}.
$$

The \emph{symmetrized
polydisc} $\G_n$ is defined by
$$\G_n:=\{\pi(A):A\in\Om_n\},$$
where the mapping $\pi : \mathcal M_n \longrightarrow \C^n$, $\pi=(\sigma_1,\dots,\sigma_n)$,
is given,
up to alternating signs,
by the coefficients of the characteristic polynomial of the matrix:
$$
P_A (t):=\det(tI_n-A)=:\sum_{j=0}^n(-1)^j\sigma_j(A)t^{n-j}.
$$
In other words, the $k$-th coordinate of $\pi$,
$\sigma_k(A)$, is the $k$-th elementary symmetric function
of the eigenvalues of $A$.
\end{defn}

\begin{prob} (The Lifting Problem).

\label{liftprob}
Given a map $\varphi\in \hol(\D,\G_n)$  and $A_1, \dots, A_N \in \Omega_n$,
find conditions (necessary, or sufficient)
such that there
exists a $\Phi\in\hol(\D,\Om_n)$ satisfying $\varphi=\pi\circ\Phi$
and $\Phi(\alpha_j)=A_j$ for $j=1,\dots,N$.
\end{prob}

When this happens, we say that the map
$\varphi$ \emph{lifts} through
the matrices $A_1, \dots, A_N$ at $(\alpha_1, \dots, \alpha_N)$.
An obvious necessary condition for $\varphi$ to lift through
the matrices $A_1, \dots, A_N$ at $(\alpha_1, \dots, \alpha_N)$ is that
$\varphi(\alpha_j)=\pi(A_j)$ for $j=1,\dots,N$.
\vskip.3cm

\begin{remark}
\label{similar}
Whenever there is a solution
to the Lifting Problem for
$\alpha_j,A_j$, then there is one for $\alpha_j,\tilde A_j$, when $A_j \sim \tilde A_j$ for each $j$,
i.e.
$A_j$ is similar to $\tilde A_j$, i.e.
for each $j$ there exists $P_j \in \mathcal M_n^{-1}$
such that
$\tilde A_j=P_j^{-1} A_j P_j$ \cite[proof of Theorem 2.1]{Agl-You1}.
\end{remark}

Our first result is a local answer (Proposition \ref{localnasc} in Section \ref{NC}).
Using this and Forstneri\v c theory, Andrist \cite{An} recently gave a proof that the local conditions
are indeed sufficient for a global lifting.
However, we provide
an explicit formula to perform the lifting when $n\le 5$ (Theorem \ref{lift4} in Section \ref{N4}).
The method developed in Section \ref{GML}
works in a number of other cases (for instances when each of the matrices
to interpolate has a single eigenvalue), but fails in general
for dimensions greater or equal to $6$, as is shown in Section \ref{cex}.

\subsection{Motivations}

Lifting maps reduces the study of Nevanlinna-Pick
interpolation into the spectral ball to a problem with a target domain of much smaller
dimension, and bounded. The symmetrized polydisc is \emph{taut},
i.e. any family of maps into it is a normal family. So
in particular, if the conditions for lifting are continuous in $\varphi$
(for instance depending on a finite number of values of $\varphi$ and its
derivatives), we can derive continuity results.
In order to formalize this, we use the following notation.

\begin{defn}
Given a map $\varphi: \D \longrightarrow \C^n$ and $\alpha \in \D$, the \emph{$k$-jet}
of $\varphi$ at $\alpha$ is defined as $ \mathcal J^k_\alpha (\varphi):=
\left( \varphi (\alpha), \varphi' (\alpha),  \dots, \varphi^{(k)} (\alpha) \right) \in \C^{(k+1)n}$.
\end{defn}

The map $\varphi \mapsto \mathcal J^k_\alpha (\varphi)$ is linear and continuous
from $\hol (\D,\C^n)$, endowed with the topology of uniform convergence on compacta,
to $\C^{(k+1)n}$.

We give an instance of a continuity result when $N=2$. Recall the definition of the Lempert function
in this context.

\begin{defn}
\label{lempfunc}
Let $A,B \in \Omega \subset \C^m$, the \emph{Lempert function} is
$$
\ell_{\Omega} (A,B):= \inf \left\lbrace |\alpha|: \exists \varphi \in \hol(\D,\Omega):
\varphi(\alpha)=A, \varphi(0)= B \right\rbrace.
$$
\end{defn}

\begin{defn}
We say that a matrix $A \in \mathcal M_n (\C)$ is \emph{cyclic} if it admits a cyclic
vector $v$, i.e. $v\in \C^n$ such that its iterates $\{ A^k x,  k\ge 0\}$ span the whole
space $\C^n$.
\end{defn}

It is well-known that the Lempert function is continuous in both arguments
at points $(A,B)$ where both $A$ and $B$ are cyclic. The situation is not as clear
when one of the matrices fails to be cyclic (such matrices are called \emph{derogatory}).

The following proposition is implicit in \cite[Section 4]{Tho-Tra}.

\begin{prop}
\label{thetacond}
Let $A, B \in \mathcal M_n (\C)$, $\alpha \in \D \setminus \{0\}$.

Suppose that $A$ is a cyclic matrix.

Suppose that there exist integers $k, p \ge 0$ and
 a linear map 
 \begin{eqnarray*}\Theta_B : \C^{(k+1)n} &\longrightarrow &\C^n \times \C^p \\
 V= (V_0, \dots, V_k) & \mapsto &\left( V_0, \tilde \Theta_B (V) \right) 
 \end{eqnarray*}
 such that: a map $\varphi \in \hol (\D)$ admits a lifting $\Phi$ through $(A,B)$
 at $(\alpha,0)$, with $\Phi(\zeta)$ cyclic for $\zeta \neq 0$,
  if and only if $\varphi(\alpha) = \pi (A)$ and
 $\Theta_B (\mathcal J^k_0 (\varphi)) = \left(\pi (B), 0 \right)$.

Then the map
 $M \mapsto \ell_{\Omega_n} (M,B)$ is continuous at the point $A$.
\end{prop}

In less technical terms, if $A$ is cyclic, and if $B$ is such that
the existence of
a lifting of a map $\varphi$ from the symmetrized polydisc through $(A,B)$, 
cyclic-valued except for $B$,  is characterized by a finite number of conditions on the values
of $\varphi$ and its derivatives (at its relevant points),
then we have
partial continuity of the Lempert function with respect to the first argument at $(A,B)$.

Since Andrist's work \cite{An} (and, in a special case, our Theorem \ref{lift4})
provides a set of conditions as called for in Proposition \ref{thetacond},
we now know that the conclusion holds for any matrix $B \in \mathcal M_n$ 
and any dimension $n$: if $A$ is cyclic,
the map
 $M \mapsto \ell_{\Omega_n} (M,B)$ is continuous at the point $A$.

\begin{proof*}{\it Proof of Proposition \ref{thetacond}.}

The Lempert function is always upper semicontinuous, so we only need to show:
$$
\mbox{when } A_p \to A, \quad\ell_{\Omega_n}(A,B) \le \limsup_p \ell_{\Omega_n}(A_p,B).
$$
Passing to a subsequence if needed,
we may choose $\alpha_p \in \D$ such that $|\alpha_p| \ge \ell_{\Omega_n}(A_p,B)$
and $\lim_{p\to\infty} |\alpha_p| = \left| \alpha_\infty \right| 
= \limsup_{p\to\infty} \ell_{\Omega_n}(A_p,B)$,
with $\alpha_\infty \in \overline \D$.
Then there exist $\Phi_p \in \hol(\D,\Omega_n )$ such that
$\Phi_p (0) = B$, $\Phi_p (\alpha_p)= A_p$.

Let $\varphi_p := \pi \circ \Phi_p$.
Because $\G_n$ is taut, passing to a subsequence if needed, we may assume that
$\varphi_p \to \varphi_\infty \in \hol (\D,\G_n)$.  Clearly $\varphi_\infty (0)= \pi (B)$.

By continuity of the jet map, $(\pi(B),0) = \Theta_B (\mathcal J^k_0(\varphi_p))
=\Theta_B (\mathcal J^k_0(\varphi_\infty))$,
so there exists $\Phi_\infty$ such that $\pi \circ \Phi_\infty =  \varphi_\infty$,
$\Phi_\infty(0)=B $.

Furthermore, $\pi (\Phi_\infty(\alpha_\infty))=\varphi_\infty (\alpha_\infty)$
$= \lim_p \varphi_p (\alpha_p) = \lim_p \pi(A_p) = \pi(A)$, and since $\Phi_\infty(\alpha_\infty)$
and $A$ have the same spectrum and are cyclic, $\Phi_\infty(\alpha_\infty) \sim A$.
\end{proof*}

\section{First reductions of the problem}
\label{firstred}

We need to establish some notations.

\begin{defn}
\label{asspol}
Given a vector $v:=(v_1, \dots, v_n) \in \C^{n}$, we denote
$P_{[v]}(t):= t^n + \sum_{j=1}^n(-1)^j v_j t^{n-j}$.
\end{defn}
This choice ensures that $P_{[\pi(A)]} = P_A$.

\begin{defn}
Given $a:=(a_1, \dots, a_n) \in \C^n$, the \emph{companion matrix}
of $a$ is
$$
C_{[a]} :=
\left(
\begin{array}{ccccc}
0 & 1 & 0 & \cdots & 0 \\
0 & 0 & \ddots & & \vdots \\
\vdots & & \ddots & 1 & 0 \\
0 & 0 & \cdots & 0 & 1 \\
a_n & a_{n-1} & \cdots & a_2 & a_1
\end{array}
\right) .
$$
\end{defn}
We see that its characteristic polynomial is then
$\det (tI_n - C_{[a]})= t^n - \sum_{j=1}^n  a_j t^{n-j} $,
so that $\sigma_j(C_{[a]}) = (-1)^{j+1} a_j$, for $1\le j \le n$.

Given a matrix $M$, the \emph{companion matrix} of $M$, denoted $C_M$, is the unique matrix in companion
form with the same characteristic polynomial as $M$.

A matrix $A\in\M_n$ is cyclic (i.e. non-derogatory) if and only if it is conjugate
to its companion matrix (for this and other equivalent properties see for instance \cite{NTZ}).

The above computation of the characteristic polynomial of a companion matrix
shows that $\varphi \in \hol(\D, \G_n)$, if we write
$\tilde \varphi := ((-1)^{j+1} \varphi_j, 1\le j \le n)$,
then the map given by $\Phi(\zeta) := C_{[\tilde \varphi (\zeta)]}$ is a lifting of $\varphi$.

Therefore, in view
of Remark \ref{similar}, this means that lifting through a set of cyclic matrices
can be achieved as soon as the obvious necessary conditions $\varphi (\alpha_j)=\pi(A_j)$,
$1 \le j \le N$, are satisfied \cite[Theorem 2.1]{Agl-You1},
\cite[Theorem 2.1]{Cos}.

The case where $A_1$ has only one eigenvalue, and $A_2, \dots, A_N$ are cyclic, has been
studied in \cite{Tho-Tra}.

\section{Necessary conditions}
\label{NC}

Let $A_1 \in \mathcal M_n $. Up to conjugacy, we may assume that it is in Jordan form.
Write this
in blocks associated to each of the \emph{distinct} eigenvalues of $A_1$,
denoted $\lambda_k, 1 \le k \le s$ where $s\le n$.  Namely
\begin{equation}
\label{jf}
A_1 = \left(
\begin{array}{ccc}
B_1 &  &  \\
 & \ddots & \\
 & & B_s
\end{array}
\right) ,  \quad B_k \in \mathcal M_{m_k}, \quad \sum_{k=1}^s m_k =n,
\end{equation}
where $\mbox{Sp }B_k=\{\lambda_k\}$, and  $\lambda_j \neq \lambda_k$ for $k\neq j$.

Temporarily, we fix $k$ and write $(B,\lambda,m)$ instead of $(B_k,\lambda_k,m_k)$.
We need to set up some notation as in  \cite{Tho-Tra}.
Let $B=(b_{i,j})_{1\le i,j\le m}$. Then $b_{jj}=\lambda$,
$b_{j-1,j} \in \{0,1\}$, $2\le j \le m$, and $b_{ij}=0$ if either $i>j$ or $i+1<j$.

Let $r$ stand for the rank of $B- \lambda I_m$, so there are exactly
exactly $m-r$ columns in  $B- \lambda I_m$ which are identically zero,
the first, and the ones
indexed by the integers $j \ge 2$ such that $b_{j-1,j}=0$.
Enumerate the (possibly empty) set of column indices where
the coefficient  $ b_{j-1,j}$ vanishes as
$$
\{ j : b_{j-1,j}=0 \} =: \{b_2, \dots, b_{m-r}\},  2\le b_2 < \dots < b_{m-r} \le m.
$$
Equivalently, $b_{l+1}-b_l$ is the size of the Jordan block $B^{(l)}:= (b_{ij})_{b_l\le i,j \le  b_{l+1}-1}$.
The integer $b_{l+1}-b_l$ is also the order of nilpotence of the block $B^{(l)}$.

We choose the Jordan form so that $b_{l+1}-b_l$ is increasing for $1\le l \le m-r$,
with the convention $b_{m-r+1}:=m+1$.  It means that the possible zeroes appear
for the smallest possible indices $j$, globally.

\begin{defn}
\label{di}
For $1\le i \le m$,
$$
d_i(B)= d_i:= 1+ \# \{ k : m-i+2 \le b_k \le m \}.
$$
\end{defn}

Equivalently, $d_i-1$ is the number of columns which are identically zero, among the
last $i-1$ columns of $A_1$, or 
$$
d_i = 1+(m-r) - \max \{ j : b_j \le m-i+1\},
$$
with the agreement that the maximum equals $0$ if the set on the right hand side is empty.

One can also
interpret $d_j = d_j(B)$ as the least integer $d$ such that there is a set $S$ of
$d$ vectors in $\C^n$ with the property that the the iterates of $S$ by $B$ span
a subspace of $\C^n$ of dimension at least $j$
(we shall not need this characterization, so we do not include a proof).

Notice that $B$ is cyclic if and only if $d_i(B)=1$, for any $i$ ($b_{j-1,j}=1$
for any $j$); while it is scalar if and only if $d_i(B)=i$, for any $i$ ($b_{j-1,j}=0$
for any $j$).

The following proposition gives a set of conditions for lifting which are
locally necessary and sufficient.  This says in particular that all possible
necessary conditions that can be obtained from the behavior of $\Phi$ in a
neighborhood of $\alpha \in \D$ are exhausted by \eqref{neccond}.
\begin{prop}
\label{localnasc}
Let $\varphi \in Hol(\omega,\G_n)$, where $\omega$ is a neighborhood of $\alpha \in \D$.
Let $A_1$ be as in \eqref{jf}.  Then the following assertions are equivalent:
\begin{enumerate}[(a)]
\item
There exists $\omega' \subset \D$ a neighborhood of $\alpha$ and
$\Phi \in \hol(\omega',\Omega_n)$ such that
$$\pi \circ \Phi =  \varphi, \Phi(0)=A_1 ;
$$
\item
The map $\varphi$ verifies
\begin{equation}
\label{neccond}
\frac{d^{k}P_{[\varphi (\zeta)]}}{dt^{k}} (\lambda_j)
 = O((\zeta-\alpha)^{d_{m_j-k}(B_j)}), \quad 0 \le k \le m_j-1, 1 \le j \le s,
\end{equation}
where the $d_i$ are as in Definition \ref{di}.
\item
There exists $\omega' \subset \D$ a neighborhood of $\alpha$ and
$\Phi \in Hol(\omega',\Omega_n)$ such that
$$\pi \circ \Phi =  \varphi, \Phi(0)=A_1 \mbox{ and }
\Phi(\zeta)\mbox{ is cyclic for }\zeta\in \omega' \setminus \{ \alpha\}.
$$
\end{enumerate}
\end{prop}
Notice that the condition
$$\frac{d^{k}P_{[\varphi (\zeta)]}}{dt^{k}} (\lambda_j)
 = O(\zeta-\alpha), \quad 0 \le k \le m_j-1, 1 \le j \le s,
$$
says exactly that $P_{[\varphi (\alpha)]}(t)= P_{A_1}(t)$, in other words, $\varphi(\alpha)=\pi(A_1)$,
which is the obvious necessary condition for the existence of a lifting; and is the only condition
that is needed when $A_1$ is cyclic, that is to say when $d_{m_j-k}(B_j)=1$
for all $j$ and $k$.

Clearly condition (c) implies condition (a), so we will only prove that (a) implies (b)
(necessary conditions) and that (b) implies (c) (sufficient conditions).

\begin{proof}
First, since this is a local result, it is no loss of generality to assume that $\alpha=0$.

\vskip.3cm
{\bf Necessary conditions.}

First consider the case where there is only one eigenvalue $\lambda_1$ for $A_1$,
i.e. $s=1$, and furthermore $\lambda_1=0$.  This is settled by \cite[Corollary 4.3]{Tho-Tra},
which can be restated as follows.
\begin{lemma}
\label{nilp}
If $ \varphi = (\varphi_1, \dots, \varphi_n)
=\pi \circ  \Phi  $ with $\Phi  \in \O(\D,\Omega_n)$,
$\Phi (0)=A_1$ as in \eqref{jf}, $\sp A_1=\{0\}$ then $\varphi_i(\zeta) =O(\zeta^{d_i})$,
where the $d_i$ are as in Definition \ref{di}.
\end{lemma}

Notice that if we write $P^{(k)}_{[\varphi(\zeta)]}:= \frac{d^{k}P_{[\varphi (\zeta)]}(t)}{dt^{k}} $
(the derivative of the polynomial with respect to the indeterminate $t$, not to be confused
with derivatives with respect to the holomorphic variable $\zeta$), the conditions above can be
written as $P^{(k)}_{[\varphi(\zeta)]}(0) = O(\zeta^{d_{n-k}})$, $0\le k \le n-1$.

If  $\mbox{Sp }A_1=\{\lambda\}$,
then $\mbox{Sp }(A_1-\lambda I_n) =\{0\}$. One sees immediately that
$P_{A_1-\lambda I_n}(t)=P_{A_1}(t-\lambda)$, so that the necessary condition in
the more general case of a matrix with a single eigenvalue becomes:

\begin{lemma}
\label{singlelambda}
If $ \varphi = (\varphi_1, \dots, \varphi_n)
=\pi \circ  \Phi  $ with $\Phi  \in \hol(\D,\Omega_n)$,
$\Phi (0)=A_1$, $\sp A_1=\{\lambda\}$ then
$P^{(k)}_{[\varphi(\zeta)]}(\lambda) = O(\zeta^{d_{n-k}})$, $0\le k \le n-1$.
\end{lemma}

Now consider the general case.  To prove \eqref{neccond} for each $j$, without
loss of generality,
study the block $B_1$ associated to the eigenvalue $\lambda_1$.

\cite[Lemma 3.1]{TTZ} and the remarks following it
give a holomorphically varying factorization in some neighborhood $\omega$ of $0$ of the
characteristic polynomial of $\Phi(\zeta)$:
for $\zeta \in \omega$,
$P_{[\varphi (\zeta)]}(X) = P_\zeta^1 (X) P_\zeta^2 (X) $,
and a corresponding splitting of the space $\C^n$ in a varying direct sum
of subspsaces of dimensions $m_1$ and $n-m_1=m_2+\cdots+m_s$,
with maps $\Phi_1 \in \hol (\omega, \Omega_{m_1})$,
$\Phi_2 \in \hol (\omega, \Omega_{m_2+\cdots+m_s})$,
such that $P_\zeta^i (t)$ is the characteristic polynomial of $\Phi_i(\zeta)$,
$i=1,2$;
$P_0^1 (t)=P_{B_1}(t)=(t-\lambda_1)^{m_1}$; and
\newline
$P_0^2 (t)= (t-\lambda_2)^{m_2} \cdots (t-\lambda_s)^{m_s}$.

As before, we may consider $P_{[\varphi (\zeta)]}(t-\lambda_1)$ to reduce ourselves
to the case $\lambda_1=0$.  The proof of the necessity of \eqref{neccond} concludes
with the following lemma (applied to $k_j:= m_1-d_j(B_1)$).

\begin{lemma}
\label{prodpol}
Let $P_\zeta^0, P_\zeta^1, P_\zeta^2$ be polynomials depending holomorphically on $\zeta$,
$P_\zeta^i(t) = \sum_{j=0}^{m_i} a_j^i t^{j}$, $i=0,1,2$,
such that $P_\zeta^0(t)= P_\zeta^1(t) P_\zeta^2(t)$,
$P_0^1(t)= t^{m_1}$ and $a_{0}^2(0)=P_0^2(0)\neq 0$.

Let $(k_j, 0\le j \le m_1-1)$ be a decreasing sequence of positive integers.
Then $a_j^0 = O(\zeta^{k_j}), 0\le j \le m_1-1$ if and only if
$a_j^1 = O(\zeta^{k_j}), 0\le j \le m_1-1$.
\end{lemma}
\begin{proof}
Since $P^0$ is the product of the other two polynomials, $a_j^0= \sum_{l=0}^j a^1_{j-l}a^2_l$.
Since $(k_j)$ is decreasing,
the hypothesis $a^1_{j-l}=O(\zeta^{k_{j-l}})$ implies $a^1_{j-l}=O(\zeta^{k_{j}})$,
so $a_j^0=O(\zeta^{k_j})$.

Conversely, proceed by induction on $j$.  For $j=0$, $a_0^1 = a_0^0/a_0^2 = O(\zeta^{k_{0}})$
since the denominator does not vanish for $\zeta=0$.  Suppose the property is satisfied
for $0\le j' \le j-1$. We have
$a_j^1 = \frac1{a_0^2}\left( a_j^0 - \sum_{l=1}^{j} a^1_{j-l}a^2_l \right)$,
so by induction hypothesis $a^1_{j-l} = O(\zeta^{k_{j-l}})=O(\zeta^{k_j})$ because $(k_j)$ is decreasing, and since $a_j^0=O(\zeta^{k_j})$, we are done.
\end{proof}

{\bf Sufficient conditions.}

Using \cite[Lemma 3.1]{TTZ} and the remarks following it,
applied repeatedly,
we find some neighborhood of $0$, $\omega$, and
a holomorphically varying factorization
into mutually prime factors $P_{[\varphi(\zeta)]} (t)
= P_{[\varphi_1(\zeta)]} (t) \cdots P_{[\varphi_s(\zeta)]} (t) $, and
a splitting of the space $\C^n$,
$$
\C^n = \bigoplus_{j=1}^s \ker P_{[\varphi_j(\zeta)]} (\Phi(\zeta)).
$$
Reducing $\omega$ if needed, for each $i$, $\varphi_i \in \hol (\omega, \G_{m_i})$.
By Lemma \ref{prodpol}, it verifies the necessary conditions of vanishing
relating to the matrix $B_i \in \mathcal M_{m_i}$, with $\mbox{Sp }B_{i}= \{\lambda_i\}$.
So by Lemma \ref{singlelambda}, there exists $\Phi_i \in \hol (\omega, \Omega_{m_i})$
such that $\varphi_i = \pi \circ \Phi_i$. The ``block mapping"
$$
\Phi = \left(
\begin{array}{ccc}
\Phi_1 &  &  \\
 & \ddots & \\
 & & \Phi_s
\end{array}
\right)
$$
yields the desired lifting.
\end{proof}

\section{A formula for lifting}
\label{GML}
\subsection{A modified Jordan form}

First we need a linear algebra lemma giving us
a canonical form for matrices, slightly different from the Jordan form and adapted to our purposes.
We will call it \emph{modified Jordan form.}

We need to set some slightly modified notations. Let $A_1 \in \mathcal M_n$,
with $Sp(A_1)=\{\lambda_1, \dots , \lambda_n\}$: here the
eigenvalues are repeated according to their multiplicities.
Let $m_1, \dots, m_s$ be the respective multiplicities, set $n_j = m_1+\cdots+m_j$.
Therefore $0=n_0<n_1 <\cdots <n_s=n$ and
$\lambda_k=\lambda_{k'}$ if and only if there exists $i \in \{1,\dots,s\}$
such that $n_{i-1} < k, k' \le n_i$.

We assume that  $A_1:= (a_{ij})_{1\le i,j \le n}$
is in  Jordan form with the notations of \eqref{jf}, except
for the labeling of the eigenvalues: now $\mbox{Sp }B_k=\{\lambda_{n_k}\}$.


\begin{lemma}
\label{mjf}
A matrix $A_1$ as given above is conjugate to $A'= (a'_{ij})_{1\le i,j \le n}$
where $a'_{n_i,1+n_i}=1\neq a_{n_i,1+n_i}=0$, $1\le i \le s-1$,
and $a'_{ij}=a_{ij}$ for all other values of the indices.
\end{lemma}

Notice that this means that $a'_{ij}=0$ if $j\notin \{i,i+1\}$,
that $a'_{i,i+1}\in \{0,1\}$ and
that if $a'_{i,i+1}=0$, then $a'_{ii}=a'_{i+1,i+1}\in Sp(A_1)$.

Intuitively, at the junction of two consecutive blocks $B_k$, we
change the coefficient just above the diagonal (and outside the blocks) from $0$ to $1$.

The new basis that we will find will no longer split the space into
invariant subspaces, but we still obtain a triangular form.
What we gain is that $A'$ is cyclic if and only if $a'_{i,i+1}=1$ for all $i$,
$1\le i \le n-1$.

\begin{proof}
Let $\{e_j, 1\le j \le n\}$ be the ordered basis in which the original Jordan form is given.
Write $V_i:= \mbox{Span} \{e_j, n_{i-1} < j \le n_i\}$ for the generalized eigenspace for
the eigenvalue $\lambda_{n_i}$, in other words
$V_i=\ker (\lambda_{n_i}I_n -A_1)^n$.  Finally, let $u$ be the linear mapping associated to $A_1$.

The result will be a consequence of the following property,
 to be proved by induction on $j$, $1\le j \le n$:
$$
(P_j): \exists v_j \in \bigoplus_{i: n_i<j} V_i \mbox{ s.t. if } e'_j:= e_j + v_j,
\mbox{ then } u(e'_j)= \lambda_j e'_j + a'_{j-1,j} e'_{j-1},
$$
with $a'_{j-1,j}$ as defined in the statement of the Lemma. (Here we understand that $e'_0=0$
and an empty sum of subspaces is $\{0\}$).

The matrix $A'$ will be the matrix of $u$ in the basis $(v_1, \dots, v_n)$.

We prove $(P_j)$ by induction. $(P_1)$ is trivially satisfied with $v_1 =0$.

Now assume that $j\ge 2$, and that $(P_{j'})$ holds for $j'<j$. We are looking for $v$ such that
$$
u( e_j+v ) = \lambda_j (e_j + v) + a'_{j-1,j} (e_{j-1} + v_{j-1}),
$$
or equivalently, since $\left( u- \lambda_j \right) e_j = a_{j-1,j} e_{j-1}$,
\begin{equation}
\label{skewjordan}
\left( u- \lambda_j \right) (v) = a'_{j-1,j}  v_{j-1} + (a'_{j-1,j}- a_{j-1,j})e_{j-1}.
\end{equation}

{\bf Case 1.} There exists $i$ such that $j=n_i+1$.

Then $a_{j-1,j}=0$ and $a'_{j-1,j}=1$, so
\eqref{skewjordan} becomes $\left( u- \lambda_{n_{i+1}} \right) (v)= v_{n_i}+e_{n_i}$.
This last vector is in $W_i:=\bigoplus_{i'\le i} V_{i'}$, which is
stable under the linear map $u- \lambda_{n_{i+1}} $.  Furthermore,
the map $(u- \lambda_{n_{i+1}})|_{W_i}$
does not admit $0$ as an eigenvalue, so the equation admits a (unique) solution $v=:v_{j+1}$
in $W_i$, which is the required space since $n_i<j$, q.e.d.

{\bf Case 2.} For all $i$, $j-1\neq n_i$.

Then $a'_{j-1,j}=a_{j-1,j}$, so
\eqref{skewjordan} becomes $\left( u- \lambda_j \right) (v)= a'_{j-1,j}  v_{j-1}$,
and $\bigoplus_{i: n_i<j-1} V_i = \bigoplus_{i: n_i<j} V_i$. Again, that subspace is stable
under $u- \lambda_{j} $, the restriction of the map is a bijection, so we get
a (unique) solution $v=:v_{j+1}\in \bigoplus_{i: n_i<j} V_i$.
\end{proof}

From now on, using Remark \ref{similar}, we assume that
the matrices $A_1, \dots, A_N$ which we want to lift through are in modified Jordan form,
as defined in Lemma \ref{mjf}.

\subsection{Divided Differences}

\begin{defn}
For a polynomial $P$, the \emph{divided differences} are given recursively by:
\begin{multline*}
\Delta^0 P = P, \Delta^1 P(x_1,x_2)= \frac{P(x_1)-P(x_2)}{x_1-x_2}, \dots,
\\
\Delta^{m}P (x_1,\dots ,x_{m+1}) =
\frac{\Delta^{m-1}P (x_1,\dots ,x_{m})- \Delta^{m-1}P (x_2,\dots ,x_{m+1})}{x_1-x_{m+1}}.
\end{multline*}
\end{defn}

Recall that $\Delta^{m}P (x,\dots ,x) = \frac1{m!} P^{(m)}(x)$.  A good general
reference about divided differences is \cite{DeB}.

\subsection{A meromorphic lifting}

The following formula gives a ``meromorphic" solution to the lifting problem: some of
the matrix coefficients are given by quotients which may have poles.  Of course,
when the singularities are removable, we extend the functions in the usual way,
and the properties claimed below extend by continuity.

\begin{prop}
\label{genform}
 Let $\Phi$ be the map from $\D$ (except for some singularities)
  to $\mathcal M_n$ defined by
$$
\Phi(\zeta)  := \left(
\begin{array}{ccccc}
\phi_{1,1}(\zeta) & f_2(\zeta) & 0 & \cdots & 0 \\
0 & \phi_{2,2}(\zeta) & \ddots & & \vdots \\
\vdots & & \ddots & f_{n-1}(\zeta) & 0 \\
0 & 0 & \cdots & \phi_{n-1,n-1}(\zeta) & f_{n}(\zeta) \\
\phi_{n,1} & \phi_{n,2} & \cdots & \phi_{n,n-1} &  \phi_{n,n}
\end{array}
\right) ,
$$
where the $f_k$, $2\le k \le n$, and $\phi_{k,k}$, $1\le k \le n-1$, are holomorphic functions to be chosen,
the $f_k$ are not identically zero, and
where
\begin{equation}
\label{lastrow}
\phi_{n,\ell}:=
- \frac{\Delta^{\ell-1}P_{[\varphi(\zeta)]}(\phi_{1,1}, \dots \phi_{\ell,\ell})}{\prod_{k=\ell+1}^n f_k(\zeta)} , \quad 1 \le \ell \le n-1,
\end{equation}
and finally
$$\phi_{n,n} :=  - \Delta^{n-1}P_{[\varphi(\zeta)]}(\phi_{1,1} , \dots, \phi_{n-1,n-1} , 0)
=\varphi_1 - (\phi_{1,1} + \cdots + \phi_{n-1,n-1}).
$$

Then  $\pi \circ \Phi(\zeta) = \varphi(\zeta)$, for the values of $\zeta$ where the quotients make sense.

Let $A'$ be as in the conclusion of Lemma \ref{mjf}, and $\alpha \in \D$.
If  $f_k(\alpha)=a'_{k-1,k}$, $2\le k \le n$, $\phi_{kk}(\alpha)= \lambda_k$,
$1\le k \le n$, and
$\phi_{n,k}(\alpha)=0$, $1\le k \le n-1$, then $\Phi(\alpha)=A'$.

If $f_k (\zeta) \neq 0 $ for any $k \in \{2,\dots,n\}$ and $\zeta \in \D\setminus \{\alpha\}$,
then $\Phi(\zeta)$ is cyclic for $\zeta \in \D\setminus \{\alpha\}$.
\end{prop}
Note that we sometimes omit the argument $\zeta$ in $\phi_{ij}(\zeta)$ and other
functions.  This will happen again.

\begin{proof}
The last statement is immediately verified.

To see that $\pi \circ \Phi(\zeta) = \varphi(\zeta)$,
we compute $\det (t I_n - \Phi(\zeta))$ by expanding with respect to the last row:
\begin{multline*}
=  (t- \phi_{n,n}) \prod_{i=1}^{n-1} (t-\phi_{i,i})
-\sum_{\ell = 1}^{n-1} (-1)^{n+\ell} \phi_{n,\ell} \prod_{i=1}^{\ell-1} (t-\phi_{i,i})
\prod_{i=\ell+1}^n (-f_i)
\\
= t \prod_{i=1}^{n-1} (t-\phi_{i,i})
+ \Delta^{n-1}P_{[\varphi]}(\phi_{1,1} , \dots, \phi_{n-1,n-1} , 0) \prod_{i=1}^{n-1} (t-\phi_{i,i})
\\
+ \sum_{j = 0}^{n-2} \Delta^{j}P_{[\varphi(\zeta)]}(\phi_{1,1}, \dots \phi_{j+1,j+1}) \prod_{i=1}^{j} (t-\phi_{ii})
\\
= P_{[\varphi(\zeta)]} (t),
\end{multline*}
by Newton's formula applied at the $n$ points $(\phi_{1,1} , \dots, \phi_{n-1,n-1} , 0)$.
\end{proof}

\subsection{Preliminary Computations}

\begin{lemma}\label{formule_difference}
 For natural numbers $k\le j,$ the divided difference $\Delta^k t^j$ is given
 by the formula $$\Delta^k t^j (x_1,x_2,\ldots, x_{k+1}) = \sum_{\substack{i_1,\ldots, i_{k+1}\geq 0\\ i_1+\ldots + i_{k+1}=j-k}}x_1^{i_1}x_2^{i_2}\ldots x_{k+1}^{i_{k+1}}.
 $$
\end{lemma}

\begin{lemma}\label{lemma_ord} Let $A_1$ be as in \eqref{jf}.
Let $\varphi\in \hol (\D,\G_n)$ satisfy the
conditions \eqref{neccond} at $\alpha = 0.$

\begin{enumerate}[(a)]
    \item  Let $\phi_{1,1}, \phi_{2,2}, \ldots, \phi_{m_1,m_1}\in\hol
(\D,\C)$ with $\phi_{i,i}(0) = \lambda_1$ for $1\leq i\leq m_1.$
Then
$$
\Delta^{k}P_{[\varphi(\zeta)]}(\phi_{1,1},\phi_{2,2},\ldots, \phi_{k+1,k+1})
=  O(\zeta^{d_{m_1-k}(B_1)})
$$
for $0\leq k\leq m_1-1.$

    \item Let $\lambda_1\neq \lambda_2.$ Suppose
    $$
    \Delta^{k}P_{[\varphi(\zeta)]}(\phi_{1,1},\phi_{2,2},\ldots, \phi_{k+1,k+1}) =
    O(\zeta^{d})
    $$
    for $0\leq k\leq m_1-1$ with $d\geq d_{m_2}(B_2).$  Let
    $\phi_{m_1+1,m_1+1}\in\hol (\D,\C)$ be any holomorphic function
    with $\phi_{m_1+1,m_1+1}(0) =\lambda_2,$ then
$$
\Delta^{m_1}P_{[\varphi(\zeta)]}(\phi_{1,1},\phi_{2,2},\ldots, \phi_{m_1+1,m_1+1}) =  O(\zeta^{d_{m_2}(B_2)}).
$$
\end{enumerate}

\end{lemma}
\begin{proof}
  To prove (a), we write $P_{[\varphi(\zeta)]}$ using the Taylor formula:
$$
P_{[\varphi(\zeta)]}(t)
= \sum_{j=0}^n \frac{P_{[\varphi(\zeta)]}^{(j)}(\lambda_1)}{j!}(t-\lambda_1)^j.
$$
for $1\leq i\leq m_1,$ put $\phi_{i,i} = \lambda_1 + \phi_i$, so
that $\phi_i(\zeta)=O(\zeta)$.  By
linearity of the divided difference operator $\Delta^k,$
\begin{align*}
&\Delta^{k}P_{[\varphi(\zeta)]}(\phi_{1,1},\phi_{2,2},\ldots,
\phi_{k+1,k+1}) =
\\ & =  \sum_{j=0}^n
\frac{P_{[\varphi(\zeta)]}^{(j)}(\lambda_1)}{j!}\Delta^{k}(t-\lambda_1)^j
(\phi_{1,1},\phi_{2,2},\ldots, \phi_{k+1,k+1})
\\ & = \sum_{j=k}^n
\frac{P_{[\varphi(\zeta)]}^{(j)}(\lambda_1)}{j!}
\sum_{\substack{i_1,\ldots, i_{k+1}\geq 0\\ i_1+\ldots +
i_{k+1}=j-k}}\phi_1^{i_1}\phi_2^{i_2}\ldots \phi_{k+1}^{i_{k+1}}
\\&= \sum_{j=k}^n \frac{P_{[\varphi(\zeta)]}^{(j)}(\lambda_1)}{j!} O(\zeta^{j-k})
= \sum_{j=k}^n O(\zeta^{d_{m_1}-j(B_1)+j-k}),
\end{align*}
by the conditions  \eqref{neccond}. It follows from Definition \ref{di}
that $d_{j+1}(B_1) \leq d_j(B_1)+1$
for $1\leq j\leq n_1-1,$ so we get, for $k\le j$, $d_{m_1-j}(B_1) \ge d_{m_1-k}(B_1)-(j-k)$,
which implies (a).

To prove (b), we use the Newton interpolation formula
\begin{multline*}
 P_{[\varphi(\zeta)]}(\phi_{m_1+1,m_1+1}) =  \\=\sum_{k=0}^{m_1} \Delta^k
P_{[\varphi(\zeta)]}(\phi_{1,1},\ldots,\phi_{k+1,k+1})\cdot
(\phi_{m_1+1,m_1+1} - \phi_{1,1})\cdots
(\phi_{m_1+1,m_1+1}-\phi_{k,k}).
\end{multline*}

To estimate the left hand side, we may switch the order of the first two
eigenvalues, and apply part (a) of the Lemma with $k=0$ (so the condition
$\phi_{m_1+1,m_1+1}(0) =\lambda_2$ is enough), thus we get
\newline
$P_{[\varphi(\zeta)]}(\phi_{m_1+1,m_1+1}) = O(\zeta^{d_{m_2}(B_2)})$.

By the hypothesis, the right hand side can be written
$$
P_{[\varphi(\zeta)]}(\phi_{1,1},\ldots,\phi_{m_1+1,m_1+1})\cdot(\phi_{m_1+1,m_1+1}(0) - \phi_{m_1,m_1}(0))
(1+O(\zeta)) + O(\zeta^d),
$$
since $\lambda_1\neq \lambda_2$
implies that $(\phi_{m_1+1,m_1+1}(0) - \phi_{m_1,m_1}(0))\neq 0$.

Solving the equation for $P_{[\varphi(\zeta)]}(\phi_{1,1},\ldots,\phi_{m_1+1,m_1+1})$
yields the result.
\end{proof}

\section{The case where $n\le 5$}
\label{N4}

\begin{theorem}
\label{lift4}
Let $n \in \N^*$, $n\le 5$,
 $A_1, \dots, A_N \in \mathcal M_n$, $\alpha_1, \dots, \alpha_N\in \D$
 and $\varphi \in \hol (\D,\G_n)$.

 Then there exists $\Phi\in\hol (\D,\Om_n)$ satisfying $\phi=\pi\circ\Phi$
and $\Phi(\alpha_j)=A_j$ for $j=1,\dots,N$ if and only if $\varphi$
satisfies the conditions \eqref{neccond} for
each $j$, with $A_j$ instead of $A_1$  and $\alpha_j$ instead
of $\alpha$, for $1\le j \le N$.

Furthermore, the values $\Phi(\zeta)$ may be chosen as cyclic matrices
when $\zeta \notin \{ \alpha_1, \dots, \alpha_n\}$.
\end{theorem}

The case $n=2$ of this theorem is covered by \cite{Agl-You1} and the case $n=3$
by \cite{Tho-Tra} and \cite{NPT}.

\begin{proof}
Proposition \ref{localnasc} proves the ``only if" part of the theorem.

To prove the ``if" part, in view of Proposition \ref{genform}, it will be enough to show that
we can choose $f_k$ and $\phi_{k,k}$ such that for each
$j$, they and the entries $\phi_{n,k}$ are defined and assume the correct values
at $\alpha_j$.

Without loss of generality (by Remark \ref{similar}), we may assume
that $A_j$ is in modified Jordan form and we write
$A_j=(a^j_{i,l})_{1\le i,l\le n}$ . A first requirement is that
$f_k(\alpha_j)= a^j_{k-1,k}$, $2\le k \le n$, $1\le j \le N$, and
$\phi_{k,k}(\alpha_j)=a^j_{k,k}=\lambda_k(A_j)$, $1\le k \le n$,
$1\le j \le N$. We also require that $f_k(\zeta)\neq 0$ unless
$\zeta=\alpha_j$ and $a^j_{k-1,k}=0$, and that all of the zeros of
each $f_k$ are simple.

{\bf Claim.} We can determine further conditions (if needed) on
a finite number of the successive derivatives of $\phi_{k,k}$
at $\alpha_j$ to ensure that the entries $\phi_{n,k}$, $1\le k \le n-1$,
 are defined and assume the value $0$
at $\alpha_j$.

If this claim holds, then polynomial interpolation, for instance, lets us find
holomorphic functions $f_k$ and $\phi_{k,k}$ satisfying all
the above requirements and the conditions in the Claim.

The remainder of the argument is devoted to the proof of the above claim.
It will be enough to work at one point $\alpha_j \in \D$, which we may take to be $0$
to simplify notations.  Likewise the matrix $A_j$ will be denoted $A_1=(a_{ij})_{1\le i,j \le n}$,
$n=4 $ or $5$,
and will be in modified Jordan form.
Recall that the conditions \eqref{neccond} already imply
that $P_{[\varphi(0)]} = P_{A_1}$.

\subsection{The case $n=4$}

We consider different cases according to the values of $(a_{12}, a_{23}, a_{34}) \in \{0,1\}^3$.
If $a_{k-1,k}=1$ for all $k$, $A_1$ is cyclic and the denominators in the formula for $\phi_{4,k}$
never vanish at $\zeta=0$; since $\phi_{k,k}(0)=\lambda_k$ which is a zero of
$P_{A_1}=P_{[\varphi(0)]}$, $\Delta^{\ell-1}P_{[\varphi(0)]}$ always vanishes at
$(\phi_{1,1}, \dots, \phi_{\ell, \ell})$ for $\ell \le 3$, and we are done.

If $a_{k-1,k}=0$ for all $k$, all the eigenvalues of $A_1$ must be equal and
the proof is (essentially) done in \cite[Proof of Proposition 4.1]{Tho-Tra}.
In what follows we always assume that $A_1$ admits
at least two distinct eigenvalues.

If there is exactly one value of $k\in \{2,3,4\}$
such that $a_{k-1,k}=0$, then we have an eigenspace of dimension
$2$, say for $\lambda_1=\lambda_2$. The corresponding generalized eigenspace,
$\ker (\lambda_1 I_4 - A_1)^4$, can be of dimension
$2$ or $3$.  Reverting momentarily to the Jordan form, in the first case the matrix splits into two
$2\times2$
blocks with distinct eigenvalues
 and we may assume $\ker(\lambda_1 I_4-A_1)= \mbox{Span}\{e_1, e_2\}$, and $\lambda_1 \notin
 \{\lambda_3, \lambda_4\}$.  In the second case, $\lambda_3=\lambda_1$
 and we may assume $\ker(\lambda_1 I_4-A_1)^2= \mbox{Span}\{e_1, e_2, e_3\}$.

 In each case, with the basis vectors being permuted as needed,
  $A_1$ admits the following modified Jordan form:
 \begin{equation}
 \label{mjf1}
A_1=
\left(
\begin{array}{cccc}
\lambda_1 & 0 & 0 & 0 \\
0 & \lambda_1 & 1 & 0 \\
0 & 0 & \lambda_3 & 1  \\
0 &  0 &  0 & \lambda_4
\end{array}
\right), \mbox{ with }  \lambda_4\neq \lambda_1.
 \end{equation}

If there are exactly two values of $k\in \{2,3,4\}$
such that $a_{k-1,k}=0$, there are two possible cases:
 in the first case, the
 two indices $k$ such that $a_{k-1,k}=0$
  are consecutive and we have an eigenspace of dimension $3$ (and since there are at least
 two distinct eigenvalues, the generalized eigenspace is equal to it), so that
 $A_1$ admits the following modified Jordan form:
 \begin{equation}
 \label{mjf2}
A_1=
\left(
\begin{array}{cccc}
\lambda_1 & 0 & 0 & 0 \\
0 & \lambda_1 & 0 & 0 \\
0 & 0 & \lambda_1 & 1  \\
0 &  0 &  0 & \lambda_4
\end{array}
\right), \mbox{ with }  \lambda_4\neq \lambda_1 .
 \end{equation}

In the second case, we have $(a_{12}, a_{23}, a_{34})=(0,1,0)$,
in which case $A_1$ admits the following modified Jordan form:
 \begin{equation}
 \label{mjf3}
A_1=
\left(
\begin{array}{cccc}
\lambda_1 & 0 & 0 & 0 \\
0 & \lambda_1 & 1 & 0 \\
0 & 0 & \lambda_3 & 0  \\
0 &  0 &  0 & \lambda_3
\end{array}
\right), \mbox{ with }  \lambda_3\neq \lambda_1 .
 \end{equation}
\vskip.3cm 
{\bf Case 1: $A_1$ as in \eqref{mjf1}.} Since
$f_3(0)=f_4(0)=1$ and $f_2$ has a simple zero at $0$, we need to
prove that
\begin{equation}
\label{4ord0}
P_{[\varphi(\zeta)]}(\phi_{1,1})=O(\zeta^2),
\end{equation}
\begin{equation}
\label{4ord1}
\Delta^{1}P_{[\varphi(\zeta)]}(\phi_{1,1},\phi_{2,2})=O(\zeta),
\end{equation}
\begin{equation}
\label{4ord2}
\Delta^{2}P_{[\varphi(\zeta)]}(\phi_{1,1},\phi_{2,2},\phi_{3,3})=O(\zeta).
\end{equation}

In this case, the conditions $\eqref{neccond}$ tell us that
$$
P_{[\varphi(0)]}(t)= (t-\lambda_1)^2(t-\lambda_3)(t-\lambda_4),
$$
and
\begin{equation}
\label{dnc1} P_{[\varphi(\zeta)]}(\lambda_1)= O(\zeta^2).
\end{equation}

Applying Lemma \ref{lemma_ord}(a) with $n_1=2$, $k=0,1$, we get \eqref{4ord0} and \eqref{4ord1},
while, when $\lambda_1 \neq \lambda_3$, we get \eqref{4ord2} 
by applying Lemma \ref{lemma_ord}(b) with $n_2=1$.
When $\lambda_1 = \lambda_3$, we see that $P_{A_1}$ admits a zero of order $3$ at $\lambda_1$ and
since $\phi_{1,1}(0)=\phi_{2,2}(0)=\phi_{3,3}(0)= \lambda_1$,
$\Delta^{2}P_{[\varphi(\zeta)]}(\phi_{1,1},\phi_{2,2},\phi_{3,3})$ vanishes when $\zeta=0$, q.e.d.

\vskip.3cm
{\bf Case 2: $A_1$ is as in \eqref{mjf2}.}

In this case, the conditions $\eqref{neccond}$ tell us that
$$
P_{[\varphi(0)]}(t)= (t-\lambda_1)^3(t-\lambda_4),
$$
and, since $d_i(B_1)=i$ for $i=1,2,3$,
\begin{equation}
\label{dnc2}
P_{[\varphi(\zeta)]}(\lambda_1)= O(\zeta^3), \quad
P'_{[\varphi(\zeta)]}(\lambda_1)= O(\zeta^2).
\end{equation}
Applying Lemma \ref{lemma_ord}(a) with $m_1=3$, we obtain
$P_{[\varphi(\zeta)]}(\phi_{1,1})=O(\zeta^3)$,
$\Delta^{1}P_{[\varphi(\zeta)]}(\phi_{1,1},\phi_{2,2})=O(\zeta^2)$,
$\Delta^{2}P_{[\varphi(\zeta)]}(\phi_{1,1},\phi_{2,2},\phi_{3,3})=O(\zeta).$
Since $f_2$ and $f_3$ have a simple zero at $0$, and $f_4(0) \neq 0$,
this is what we need for the $\phi_{4,l}$ defined in \eqref{lastrow} to vanish
at $0$, for $1\le l \le 3$.

\vskip.3cm
{\bf Case 3: $A_1$ as in \eqref{mjf3}.}

This is the most complicated case.

In this case, the conditions $\eqref{neccond}$ tell us that
\begin{equation}
\label{nct}
P_{[\varphi(0)]}(t)= (t-\lambda_1)^2(t-\lambda_3)^2,
\end{equation}
and
\begin{equation}
\label{dnc3}
P_{[\varphi(\zeta)]}(\lambda_1)= O(\zeta^2), \quad
P_{[\varphi(\zeta)]}(\lambda_3)= O(\zeta^2).
\end{equation}

Referring again to \eqref{lastrow}, we need to prove that
\begin{equation}
\label{van1}
P_{[\varphi(\zeta)]}(\phi_{1,1})=O(\zeta^3),
\end{equation}
\begin{equation}
\label{van2}
\Delta^{1}P_{[\varphi(\zeta)]}(\phi_{1,1},\phi_{2,2})=O(\zeta^2),
\end{equation}
\begin{equation}
\label{van3}
\Delta^{2}P_{[\varphi(\zeta)]}(\phi_{1,1},\phi_{2,2},\phi_{3,3})=O(\zeta^2).
\end{equation}
To obtain this, we will need to choose adequate values for $\phi'_{1,1}(0)$ and $\phi'_{2,2}(0)$.

\vskip.3cm {\bf Proof of \eqref{van1}.}

Put $\phi_{1,1} = \lambda_1 + a,$  $\phi_{2,2} = \lambda_1+b,$
where $a,b$ are functions of $\zeta \in \D$. Note
that $a^{\prime}(0) = \phi_{1,1}^{\prime}(0)$ and
$b^{\prime}(0)=\phi_{2,2}^{\prime}(0).$

As in the proof of Lemma \ref{lemma_ord}(a), we have
\begin{align*}P_{[\varphi(\zeta)]}(\phi_{1,1})
 & = \Delta^{0}P_{[\varphi(\zeta)]}(\lambda_1 +a)\\ &=
P_{[\varphi(\zeta)]}(\lambda_1) +
P^{\prime}_{[\varphi(\zeta)]}(\lambda_1)a +
P^{\prime\prime}_{[\varphi(\zeta)]}(\lambda_1)\frac{a^2}{2!} +
O(\zeta^3).
\end{align*}

By \eqref{dnc3}, we have $P_{[\varphi(\zeta)]}(\lambda_1) = O(\zeta^2),$
and
$P^{\prime}_{[\varphi(\zeta)]}(\lambda_1)a = O(\zeta^2)$.
Furthermore by \eqref{nct},
$P^{\prime\prime}_{[\varphi(\zeta)]}(\lambda_1)\neq 0$ at $\zeta=0.$
So we can choose $a^{\prime}(0)$ as a solution of a non-trivial
quadratic equation to cancel out the terms of degree $2$ in $\zeta$.

\vskip.3cm
{\bf Proof of \eqref{van2}.}

As in the proof of Lemma \ref{lemma_ord}(a)
\begin{align*}
\Delta^{1}P_{[\varphi(\zeta)]}(\phi_{1,1},\phi_{2,2}) &=
\Delta^{1}P_{[\varphi(\zeta)]}(\lambda_1 + a,\lambda_1 +b)
\\&= P_{[\varphi(\zeta)]}^{\prime}(\lambda_1) +
P_{[\varphi(\zeta)]}^{\prime\prime}(\lambda_1)\frac{ a+b}{2!} +
O(\zeta^2).
\end{align*}

We have $P_{[\varphi(\zeta)]}^{\prime}(\lambda_1) = O(\zeta)$ since
$\lambda_1$ is a double root of $P_{[\varphi(0)]},$ and
$P_{[\varphi(\zeta)]}^{\prime\prime}(\lambda_1)\neq 0$  at $\zeta
=0.$ So we can chose $b^{\prime}(0)$ to cancel out the terms of degree $1$
in $\zeta$.

\vskip.3cm
{\bf Proof of \eqref{van3}.}

Taking $\phi_{1,1}$ and $\phi_{2,2}$ as explained above, we apply Lemma
\ref{lemma_ord}(b) with $m_1=m_2=2$, and we find that
$$
\Delta^{2}P_{[\varphi(\zeta)]}(\phi_{1,1},\phi_{2,2},\phi_{3,3})=O(\zeta^2)
$$
for every $\phi_{3,3}$ with $\phi_{3,3}(0)=\lambda_2.$
\end{proof}

\subsection{The case $n = 5$}


As in the case $n = 4,$ assume that $A_1$ admits at least two distinct
eigenvalues and $A_1$ is non-cyclic, so in particular there is
at least one multiple eigenvalue.

If $A_1$ admits at least three distinct eigenvalues, $A_1$ admits the following
modified Jordan form
 \begin{equation}
 \label{mjf_general}
 A_1=
\left(
\begin{array}{ccccc}
\lambda_1 & 0 & 0 & 0 & 0 \\
0 & \lambda_1 & \ast & 0  & 0\\
0 & 0 & \lambda_3 & \ast   & 0\\
0 &  0 &  0 & \lambda_4 &1 \\ 0 & 0 & 0 & 0 & \lambda_5
\end{array}
\right), \mbox{ with }  \lambda_5\neq \lambda_1,\lambda_3,\lambda_4
 \end{equation} where the stars stand for 0 or 1.

Suppose that $\Phi(\zeta)$ is as in Proposition \ref{genform}.

 Since we always have $\Delta^3
 P_{[\varphi(\zeta)]}(\phi_{1,1},\phi_{2,2},\phi_{3,3},\phi_{4,4})=0$
 at $\zeta=0,$ and $\phi_{5,4} = -\Delta^3
 P_{[\varphi(\zeta)]}(\phi_{1,1},\phi_{2,2},\phi_{3,3},\phi_{4,4})$
 from \eqref{lastrow}, it suffices to
 consider the divided differences up to order 2, and the proof
 goes through as in the case $n =4.$  The same obtains when $\lambda_1=\lambda_3=\lambda_4
 \neq \lambda_5$.

 Therefore, in this case, it suffices to consider the case where $A_1$
 admits exactly two distinct eigenvalues, each of multiplicity at least $2$. So $A_1$ admits the
 following modified Jordan form
  \begin{equation}
 \label{mjf_general1}
A_1=
\left(
\begin{array}{ccccc}
\lambda & b_{12} & 0 & 0 & 0 \\
0 & \lambda & b_{23} & 0  & 0\\
0 & 0 & \lambda & 1   & 0\\
0 &  0 &  0 & \mu & b_{45} \\ 0 & 0 & 0 & 0 & \mu
\end{array}
\right), \mbox{ with }  \lambda\neq \mu.
 \end{equation}

 If there is a cyclic block for $\lambda$ or $\mu$, i.e $b_{12} = b_{23} = 1$
 or $b_{45}=1,$ then we can change the order of the basis vectors to
 put that cyclic block in the lower right hand corner
 of the matrix and
 reduce the proof as in the case $n =4.$ So we only need to consider
 the case where these two blocks are non-cyclic. In this case, $A_1$
 admits one of the following modified Jordan forms:
  \begin{equation}
 \label{mjf51}
A_1=
\left(
\begin{array}{ccccc}
\lambda & 0 & 0 & 0 & 0 \\
0 & \lambda & 0 & 0  & 0\\
0 & 0 & \lambda & 1   & 0\\
0 &  0 &  0 & \mu & 0 \\ 0 & 0 & 0 & 0 & \mu
\end{array}
\right),
 \end{equation}
 or  \begin{equation}
 \label{mjf52}
A_1=
\left(
\begin{array}{ccccc}
\mu & 0 & 0 & 0 & 0 \\
0 & \mu & 1 & 0  & 0\\
0 & 0 & \lambda & 0   & 0\\
0 &  0 &  0 & \lambda & 1 \\ 0 & 0 & 0 & 0 & \lambda
\end{array}
\right).
 \end{equation}

{\bf Case 1: $A_1$ as in \eqref{mjf51}.}

In this case, the conditions \eqref{neccond} tell us that
\begin{equation}\label{condition51}
P_{[\varphi(\zeta)]}(\lambda) = O(\zeta^3),
P^{\prime}_{[\varphi(\zeta)]}(\lambda) = O(\zeta^2),
P^{\prime\prime}_{[\varphi(\zeta)]}(\lambda) = O(\zeta)
\end{equation} and \begin{equation}P_{[\varphi(\zeta)]}(\mu) = O(\zeta^2),
P^{\prime}_{[\varphi(\zeta)]}(\mu) = O(\zeta).\end{equation}

We need to prove that
\begin{equation}
\label{van51} P_{[\varphi(\zeta)]}(\phi_{1,1})=O(\zeta^4),
\end{equation}
\begin{equation}
\label{van52}
\Delta^{1}P_{[\varphi(\zeta)]}(\phi_{1,1},\phi_{2,2})=O(\zeta^3),
\end{equation}
\begin{equation}
\label{van53}
\Delta^{2}P_{[\varphi(\zeta)]}(\phi_{1,1},\phi_{2,2},\phi_{3,3})=O(\zeta^2),
\end{equation}\begin{equation}
\label{van54}
\Delta^{3}P_{[\varphi(\zeta)]}(\phi_{1,1},\phi_{2,2},\phi_{3,3},\phi_{4,4})=O(\zeta^2).
\end{equation}

To obtain \eqref{van51}, \eqref{van52}, and \eqref{van53}, we will need to choose adequate values for
$\phi'_{1,1}(0),$ $\phi'_{2,2}(0),$ and $\phi^{\prime}_{3,3}(0).$
Note that as soon as we will have chosen these values, \eqref{van54}
will be automatically satisfied by Lemma \ref{lemma_ord}(b),
applied with $m_1=3$, $m_2=2$, and $d=2$.

\vskip.3cm {\bf Proof of \eqref{van51}.} Put $\phi_{1,1}=\lambda+a,$
$\phi_{2,2}=\lambda+b$ and $\phi_{3,3}=\lambda+c.$ As in the proof of
Lemma \ref{lemma_ord}(a), we write
\begin{align*} P_{[\varphi(\zeta)]}(\phi_{1,1}) &= \Delta^0
P_{[\varphi(\zeta)]}(\lambda +a)\\ & = P_{[\varphi(\zeta)]}(\lambda)
+ P^{\prime}_{[\varphi(\zeta)]}(\lambda)a +
P^{\prime\prime}_{[\varphi(\zeta)]}(\lambda)\frac{a^2}{2!} +
P^{\prime\prime\prime}_{[\varphi(\zeta)]}(\lambda)\frac{a^3}{3!} +
O(\zeta^4).
\end{align*}

From the conditions \eqref{condition51}, we find that

$$\mathrm{ord}_{\zeta=0}\left\{P_{[\varphi(\zeta)]}(\lambda)+
P^{\prime}_{[\varphi(\zeta)]}(\lambda)a +
P^{\prime\prime}_{[\varphi(\zeta)]}(\lambda)\frac{a^2}{2!}\right\}
\geq 3$$ and $P^{\prime\prime\prime}_{[\varphi(\zeta)]}(\lambda)\neq
0$ at $\zeta=0.$ So we can choose $a^{\prime}(0)$ as a solution of a
non-trivial cubic equation to cancel out the terms of degree $3$.

\vskip.3cm
 {\bf Proof of \eqref{van52}.}
We have
\begin{align*}\Delta^{1}P_{[\varphi(\zeta)]}(\phi_{1,1},\phi_{2,2}) & =
\Delta^{1}P_{[\varphi(\zeta)]}(\lambda+a,\lambda+b) \\ &=
P^{\prime}_{[\varphi(\zeta)]}(\lambda) +
P^{\prime\prime}_{[\varphi(\zeta)]}(\lambda)\frac{a+b}{2!} +
P^{\prime\prime\prime}_{[\varphi(\zeta)]}(\lambda)\frac{a^2+
ab+b^2}{3!} + O(\zeta^3).
\end{align*}
Again, we can choose $b'(0)$ to cancel out the terms of degree $2$.

\vskip.3cm
 {\bf Proof of \eqref{van53}.}
\begin{align*}\Delta^{2}P_{[\varphi(\zeta)]}(\phi_{1,1},\phi_{2,2},\phi_{3,3}) 
&= \Delta^{2}P_{[\varphi(\zeta)]}(\lambda+a,\lambda+b,\lambda+c)
\\ &= P^{\prime\prime}_{[\varphi(\zeta)]}(\lambda)\frac{1}{2!} +
P^{\prime\prime\prime}_{[\varphi(\zeta)]}(\lambda)\frac{a+b+c}{3!} +
O(\zeta^2).
\end{align*}
Again, we can choose $c'(0)$ to cancel out the terms of degree $1$.

\vskip.3cm {\bf Case 2: $A_1$ as in \eqref{mjf52}.}

In this case, the conditions \ref{neccond} tell us that
\begin{equation}\label{condition52}
P_{[\varphi(\zeta)]}(\mu) = O(\zeta^2),
P^{\prime}_{[\varphi(\zeta)]}(\mu) = O(\zeta),
\end{equation} and \begin{equation}P_{[\varphi(\zeta)]}(\lambda) = O(\zeta^2),
P^{\prime}_{[\varphi(\zeta)]}(\lambda) = O(\zeta),
P^{\prime\prime}_{[\varphi(\zeta)]}(\lambda) =
O(\zeta)\end{equation}

We need to prove that
\begin{equation}
\label{van521} P_{[\varphi(\zeta)]}(\phi_{1,1})=O(\zeta^3),
\end{equation}
\begin{equation}
\label{van522}
\Delta^{1}P_{[\varphi(\zeta)]}(\phi_{1,1},\phi_{2,2})=O(\zeta^2),
\end{equation}
\begin{equation}
\label{van523}
\Delta^{2}P_{[\varphi(\zeta)]}(\phi_{1,1},\phi_{2,2},\phi_{3,3})=O(\zeta^2),
\end{equation}\begin{equation}
\label{van524}
\Delta^{3}P_{[\varphi(\zeta)]}(\phi_{1,1},\phi_{2,2},\phi_{3,3},\phi_{4,4})=O(\zeta).
\end{equation}

\vskip.3cm
 {\bf Proof of \eqref{van521}.} We have

 \begin{align*} P_{[\varphi(\zeta)]}(\phi_{1,1}) & = \Delta^0
 P_{[\varphi(\zeta)]}(\mu + a) \\ & = P_{[\varphi(\zeta)]}(\mu) +
 P'_{[\varphi(\zeta)]}(\mu)a +
 P''_{[\varphi(\zeta)]}(\mu)\frac{a^2}{2!}+O(\zeta^3).
 \end{align*}

 From the conditions \eqref{condition52}, we find that
 $$\mathrm{ord}_{\zeta=0}\left\{P_{[\varphi(\zeta)]}(\mu) +
 P'_{[\varphi(\zeta)]}(\mu)a\right\}\geq 2 $$ and
 $P''_{[\varphi(\zeta)]}(\mu)\neq 0$ at $\zeta=0,$ so we can choose
 $a'(0)$ as a solution of a non-trivial quadratic equation to cancel out the terms of degree $2$.

 \vskip.3cm
 {\bf Proof of \eqref{van522}.}

 \begin{align*}
\Delta^{1}P_{[\varphi(\zeta)]}(\phi_{1,1},\phi_{2,2}) & =
\Delta^{1}P_{[\varphi(\zeta)]}(\mu+a,\mu+b) \\ & =
 P'_{[\varphi(\zeta)]}(\mu) +
 P''_{[\varphi(\zeta)]}(\mu)\frac{a+b}{2!}+O(\zeta^2).
 \end{align*}
Again, we can choose $b'(0)$ to cancel out the terms of degree $1$.

\vskip.3cm
 {\bf Proof of \eqref{van523}.} This follows automatically from
 \eqref{van521}, \eqref{van522} and Lemma
 \ref{lemma_ord}(b).

\vskip.3cm
 {\bf Proof of \eqref{van524}.}  As remarked after formula \eqref{mjf_general},
 $\Delta^3
 P_{[\varphi(\zeta)]}(\phi_{1,1},\phi_{2,2},\phi_{3,3},\phi_{4,4})$ always
 vanishes at $\zeta=0$.

\section{Counter-examples for the candidate lifting when $n\geq 6$}
\label{cex}

Suppose that $\Phi(\zeta)$ is given by the formula stated in Proposition
\ref{genform}, and that $\Phi(0)=A_1$, with $A_1$ as
described in Lemma \ref{mjf}.  In particular we should have $\phi_{n,1}(\zeta)=O(\zeta),$
and we would then have
\begin{align}
\label{ordineq}\nonumber \ord_{\zeta=0}
(P_{[\varphi(\zeta)]}(\phi_{1,1})) & \geq
\ord_{\zeta=0}{\phi_{n,1}}+ \ord_{\zeta=0}(f_2f_3\ldots f_n)
\\& \geq 1 + d_{m_1}(B_1)-1 + \ord_{\zeta=0} (f_{m_1+1}
f_{m_1+2}\ldots f_n)
\\
\nonumber
& \geq d_{m_1}(B_1) + \ord_{\zeta
=0} (f_{m_1+1}f_{m_1+2}\ldots f_n).
\end{align}

To look for a counter-example, we look for a situation where the
above inequality does not occur.

Consider the  following matrix of size $k$
$$
B^{\lambda}_k = \begin{pmatrix} \lambda & 0 &&&& \\ & \lambda & 0
&&&
\\ &&\ddots&\ddots&&\\ &&&\lambda&0& \\&&&&\lambda&1 \\ &&&&&\lambda
\end{pmatrix}$$ with $\lambda\in\mathbb{D}.$ As usual, we understand
that there are $0$'s in the places where no entry is indicated. Put
$$ B^{\lambda}_k(\zeta) = \begin{pmatrix} \lambda & \zeta &&&& \\ & \lambda &
\zeta &&&
\\ &&\ddots&\ddots&&\\ &&&\lambda&\zeta& \\&&&&\lambda&1 \\ \zeta&&&&&\lambda
\end{pmatrix}$$
with $\zeta \in\mathbb{D}.$ Its characteristic polynomial is given
by the formula
$$
P_{B^{\lambda}_k(\zeta)}(t) = \det(tI_n-B^{\lambda}(\zeta)) = (t-\lambda)^k -\zeta^{k-1}.
$$

Decompose $n = k+l$ with $k,l \geq 3.$

Consider the matrix $$B = \begin{pmatrix}   B^{\lambda_1}_k & \\ &
B^{\lambda_2}_l
\end{pmatrix}$$ with $\lambda_1 \neq \lambda_2\in \mathbb{D}.$  Put
$$B(\zeta) = \begin{pmatrix}   B^{\lambda_1}_k(\zeta) & \\ &
B^{\lambda_2}_l(\zeta)
\end{pmatrix}.$$
Setting $\varphi(\zeta) = (\pi\circ B)(\zeta)$, we have
$$
P_{[\varphi(\zeta)]}(t) = P_{B(\zeta)}(t) =
P_{B_k^{\lambda_1}(\zeta)}(t)P_{B_l^{\lambda_2}(\zeta)}(t).
$$
Suppose now that we can find a lifting $\Phi(\zeta)$ in the form stated
in Proposition \ref{genform}. Inequality \eqref{ordineq} gives the estimate
$$\ord_{\zeta=0}P_{[\varphi(\zeta)]}(\phi_{1,1}) \geq  1+ (k -2) +(l-2) =k+l-3\geq k.$$

Since $P_{B^{\lambda_2}_l}(\zeta)(\phi_{1,1})\neq 0$ at $\zeta=0,$
\begin{align*} \ord_{\zeta=0}P_{[\varphi(\zeta)]}(\phi_{1,1})&
= \ord_{\zeta=0} P_{B_k^{\lambda_1}(\zeta)}(\phi_{1,1}) \\ & =
\ord_{\zeta=0} \left\{(\phi_{1,1}-\lambda_1)^k -\zeta^{k-1}\right\} \\
& = k-1. \end{align*}

So this gives a contradiction. One will object that the range of
$\varphi$ may fail to be  in $\mathbb{G}_n.$ However, $\varphi(0)\in
\mathbb{G}_n,$ so we can always reparametrize $\varphi$ by replacing
$\zeta$ by $\varepsilon \zeta$ with $\varepsilon$ small enough.


\vskip.5cm
{\bf Acknowledgements.} This work was made possible in part by the CNRS-BAS research cooperation
project no. 97946, ``Holomorphic invariant of domains".  We also would like to thank the Polish
hosts of the workshop in Bedlewo where the
first two named authors started this in 2010; and the Laboratoire International
Associ\'e ``Formath Vietnam" (funded by CNRS) and the Vietnam Institute for Advanced Studies in
Mathematics in Hanoi, which helped the second and third named author collaborate on this topic.

We wish to thank the referee for a careful reading of a first version of this paper,
which helped streamline the exposition in many places.

\end{document}